\documentclass[12pt]{article}
\usepackage{graphicx}
\usepackage{inputenc}
\usepackage[ruled,vlined]{algorithm2e}
\usepackage{appendix}
\usepackage{amsmath}
\usepackage{amssymb}
\usepackage{amsthm}
\usepackage{multirow}
\usepackage{color}
\usepackage{longtable}
\usepackage{array}
\usepackage{url}
\usepackage{booktabs}
\usepackage{float}
\usepackage{mathtools}
\usepackage{tikz}
\usepackage{caption}

\newtheorem{theorem}{Theorem}[section]
\newtheorem{definition}[theorem]{Definition}
\newtheorem{lemma}[theorem]{Lemma}

\newtheorem{corollary}[theorem]{Corollary}

\newtheorem{conjecture}[theorem]{Conjecture}

\theoremstyle{definition}
\newtheorem{remark}[theorem]{Remark}

\newcommand{\turan}{Tur\'{a}n }

\title{Some extremal results on hypergraph Tur\'{a}n problems}
\author{Zixiang Xu$^{\text{a,}}$,\thanks{e-mail: zxxu8023@qq.com}~ Tao Zhang$^{\text{b,}}$\thanks{e-mail: zhant220@163.com. Research supported by the National Natural Science Foundation of China under Grant No. 11801109.}~ and  Gennian Ge$^{\text{a,}}$\thanks{e-mail: gnge@zju.edu.cn. Research supported by the National Natural Science Foundation of China under Grant No. 11971325, National Key Research and Development Program of China under Grant Nos. 2020YFA0712100 and 2018YFA0704703, and Beijing Scholars Program.}\\
\footnotesize $^{\text{a}}$ School of Mathematical Sciences, Capital Normal University, Beijing 100048, China.\\
\footnotesize $^{\text{b}}$ School of Mathematics and Information Science, Guangzhou University, Guangzhou 510006, China.\\}
\begin{document}

\date{}

\maketitle

\begin{abstract}
For two $r$-graphs $\mathcal{T}$ and $\mathcal{H}$, let $\text{ex}_{r}(n,\mathcal{T},\mathcal{H})$ be the maximum number of copies of $\mathcal{T}$ in an $n$-vertex $\mathcal{H}$-free $r$-graph. The determination of Tur\'{a}n number $\text{ex}_{r}(n,\mathcal{T},\mathcal{H})$ has become the fundamental core problem in extremal graph theory ever since the pioneering work Tur\'{a}n's Theorem was published in $1941$. Although we have some rich results for the simple graph case, only sporadic results have been known for the hypergraph Tur\'{a}n problems. In this paper, we mainly focus on the function $\textup{ex}_{r}(n,\mathcal{T},\mathcal{H})$ when $\mathcal{H}$ is one of two different hypergraph extensions of the complete bipartite graph $K_{s,t}$. The first extension is the complete bipartite $r$-graph $K_{s,t}^{(r)}$, which was introduced by Mubayi and Verstra\"{e}te~[J.
Combin. Theory Ser. A, 106: 237--253, 2004]. Using the powerful random algebraic method, we show that if $s$ is sufficiently larger than $t$, then
\[\text{ex}_{r}(n,\mathcal{T},K_{s,t}^{(r)})=\Omega(n^{v-\frac{e}{t}}),\]
where $\mathcal{T}$ is an $r$-graph with $v$ vertices and $e$ edges. In particular, when $\mathcal{T}$ is an edge or some specified complete bipartite $r$-graph, we can determine their asymptotics.
 The second important extension is the complete $r$-partite $r$-graph $K_{s_{1},s_{2},\ldots,s_{r}}^{(r)}$, which has been widely studied. When $r=3$, we provide an explicit construction giving
\[\text{ex}_{3}(n,K_{2,2,7}^{(3)})\geqslant\frac{1}{27}n^{\frac{19}{7}}+o(n^{\frac{19}{7}}).\]
Our construction is based on the Norm graph, and improves the lower bound $\Omega(n^{\frac{73}{27}})$ obtained by probabilistic method.


\smallskip

\end{abstract}

\section{Introduction}
In this paper, an $r$-graph is always an $r$-uniform hypergraph. Let $\mathcal{H}$ be an $r$-graph. An $r$-graph $\mathcal{G}$ is called $\mathcal{H}$-free if $\mathcal{G}$ contains no copy of $\mathcal{H}$ as a subhypergraph. Define $\text{ex}_{r}(n,\mathcal{T},\mathcal{H})$ to be the maximum number of copies of $\mathcal{T}$ in an $n$-vertex $\mathcal{H}$-free $r$-graph. In particular, if $\mathcal{T}$ is a single edge, then $\text{ex}_{r}(n,\mathcal{T},\mathcal{H})$ is equivalent to the classical Tur\'{a}n number $\text{ex}_{r}(n,\mathcal{H}).$ Moreover, when $r=2$, we usually use $\text{ex}(n,T,H)$ rather than $\text{ex}_{2}(n,T,H).$

The study of Tur\'{a}n numbers plays an important role in extremal graph theory. One of the oldest results on Tur\'{a}n numbers, which states that every graph on $n$ vertices with more than $\frac{n^{2}}{4}$ edges contains a triangle, was proved by Mantel~\cite{1907Mantel} in $1907$. This result was generalized later to $K_{\ell}$-free graphs by Tur\'{a}n~\cite{1941Turan}. Furthermore, the Erd\H{o}s-Stone-Simonovits theorem~\cite{1966ES, 1946ErodsBAMS} is an asymptotic version of a generalization of Tur\'{a}n's theorem, which gives the bound for the number of edges in an $H$-free graph, where $H$ is a non-complete graph. Bollob\'{a}s~\cite{1998BollobasGTM} described the Erd\H{o}s-Stone-Simonovits theorem as the ``fundamental theorem of extremal graph theory''. The determination of the exact asymptotics for $\text{ex}(n,H)$ is far from being solved when $H$ is a bipartite graph. One of the important cases is the complete bipartite graph $K_{s,t}$. A well-known result of K\"{o}vari, S\'{o}s and Tur\'{a}n~\cite{Kovari1954} showed that $\text{ex}(n,K_{s,t})=O(n^{2-\frac{1}{s}})$ for any integers $t\geqslant s$. Erd\H{o}s, R\'{e}nyi and S\'{o}s~\cite{Erdos1966} and Brown \cite{Brown1966} respectively proved matching lower bounds for the cases $s=2$ and $s=3$. For general values of $s$ and $t$, Koll\'{a}r, R\'{o}nyai and Szab\'{o}~\cite{KRS96} first showed that $\text{ex}(n,K_{s,t})=\Omega(n^{2-\frac{1}{s}})$ when $t\geqslant s!+1.$ The bound on $t$ was improved to $t\geqslant (s-1)!+1$ by Alon, R\'{o}nyai and Szab\'{o}~\cite{ARS99}. Based on some ideas in~\cite{BukhIJM2013}, Bukh~\cite{Bukh2015RAC} gave a new construction of $K_{s,t}$-free graphs which also yields a matched lower bound $\text{ex}(n,K_{s,t})=\Omega(n^{2-\frac{1}{s}}),$ where $t$ is sufficiently larger than $s$.

In contrast to the simple graph case, there are only a few results for the hypergraph Tur\'{a}n problems. For example, even the asymptotic value of $\text{ex}_{r}(n,K_{t}^{(r)})$ is still unknown for any $t>r\geqslant 3.$ In addition to complete $r$-graphs, some other cases were studied recently. Let $K_{s_{1},s_{2},\ldots,s_{r}}^{(r)}$ be a complete $r$-partite $r$-graph with parts of sizes $s_{1},s_{2},\ldots,s_{r}$, Mubayi~\cite{Mubayi2002} conjectured that $\text{ex}_{r}(n,K_{s_{1},s_{2},\ldots,s_{r}}^{(r)})=\Theta(n^{r-\frac{1}{\prod_{i=1}^{r-1}s_{i}}})$, where $s_{1}\leqslant s_{2}\leqslant\dots\leqslant s_{r}$. In the same paper, the author verified this conjecture when $s_{1}=s_{2}=\cdots=s_{r-2}=1$ and (i) $s_{r-1}=2, s_{r}\geqslant 2,$ (ii) $s_{r-1}=s_{r}=3,$ (iii) $s_{r-1}\geqslant 3, s_{r}>(s_{r-1}-1)!.$ Using the random algebraic method, Ma, Yuan and Zhang~\cite{Ma2018} showed that if $s_{r}$ is sufficiently larger than $s_{1},s_{2},\ldots,s_{r-1},$ then this conjecture is true.

For the function $\text{ex}_{r}(n,\mathcal{T},\mathcal{H})$, where $\mathcal{T}$ is not an edge, there are only sporadic results. When $r=2$, it corresponds to the classical generalized Tur\'{a}n number $\text{ex}(n,T,H)$, where $H$ and $T$ are graphs. In~\cite{Alon2016}, Alon and Shikhelman studied $\text{ex}(n,T,H)$ systematically and obtained many results on certain graphs such as complete graphs, complete bipartite graphs and trees. Later, Ma, Yuan and Zhang~\cite{Ma2018} improved some of their results. They showed that for any positive integers $a<s,$ $b\leqslant s$ and $t\geqslant f(a,b,s)$, $\text{ex}(n,K_{a,b},K_{s,t})=\Theta(n^{a+b-\frac{ab}{s}}).$  In the same paper, they also provided some bounds for $\text{ex}_{r}(n,\mathcal{T},K_{s_{1},s_{2},\ldots,s_{r-1},s_{r}}^{(r)})$ under certain conditions. For more extremal results of graphs and hypergraphs, we refer the readers to the surveys \cite{Furedi1991, Furedi2013, Keevash2015}.

In $2004$, Mubayi and Verstra\"{e}te \cite{Mubayi2004} considered a hypergraph extension of the complete bipartite graph. In this paper, we call this extension a complete bipartite $r$-graph for simplicity.
\begin{definition}[Complete bipartite $r$-graph]
  Let $X_{1},X_{2},\ldots,X_{t}$ be $t$ pairwise disjoint sets of size $r-1$, and let $Y$ be a set of $s$ elements, disjoint from $\bigcup\limits_{i\in [t]}X_{i}$. Then $K_{s,t}^{(r)}$ denotes the complete bipartite $r$-graph with vertex set $(\bigcup\limits_{i\in [t]}X_{i})\cup Y$ and edge set $\{X_{i}\cup\{y\}:i\in [t],y\in Y\}$.
\end{definition}
In \cite{Mubayi2004}, Mubayi and Verstra\"{e}te showed some bounds for $\text{ex}_{r}(n,K_{s,t}^{(r)})$ when $s\leqslant t$.
They showed $\text{ex}_{3}(n,K_{2,t}^{(3)})=\Theta(n^{2})$, and if $\frac{n}{3}\geqslant t \geqslant s\geqslant 3$, then $\text{ex}_{3}(n,K_{s,t}^{(3)})=O(n^{3-\frac{1}{s}})$. They also gave a construction which yields $\text{ex}_{3}(n,K_{s,t}^{(3)})=\Omega(n^{3-\frac{2}{s}})$ for $t>(s-1)!>0.$ In \cite{JiangJCTA2020}, Ergemlidze, Jiang and Methuku determined the expression $g(t)=\lim\limits_{n\rightarrow \infty}\frac{\textup{ex}_{3}(n,K_{2,t}^{(3)})}{\binom{n}{2}}=\Theta(t^{1+o(1)})$ as $t \rightarrow \infty $.

Note that $K_{s,t}^{(r)}$ and $K_{t,s}^{(r)}$ are nonisomorphic when $r\geqslant 3$ and $s\neq t$. The authors in~\cite{Mubayi2004} remarked that their results apply to both cases, so for simplicity they let $t\geqslant s.$  In this paper, we focus on the other case $s>t$ and $r\geqslant 3.$

Our first result gives a lower bound for $\text{ex}_{r}(n,\mathcal{T},K_{s,t}^{(r)})$ shown in the following theorem, where $\mathcal{T}$ is an arbitrary $r$-graph.

\begin{theorem}\label{thm:GTLB}
 Let $r\geqslant 3.$ For any positive integer $t$, and any $r$-graph $\mathcal{T}$ with $v$ vertices and $e$ edges, there exists some constant $c$ which depends on $r$ and $t$ such that if $s\geqslant c,$ then we have
  \begin{equation*}
    \textup{ex}_{r}(n,\mathcal{T},K_{s,t}^{(r)})=\Omega(n^{v-\frac{e}{t}}).
  \end{equation*}
\end{theorem}

To obtain the lower bound in Theorem~\ref{thm:GTLB}, our construction of $K_{s,t}^{(r)}$-free $r$-graphs is based on the random algebraic method which was introduced by Bukh~\cite{Bukh2015RAC}. Using the random algebraic method, Bukh and Conlon \cite{Bukh2018} verified the rational exponent conjecture which was presented in \cite{Erdos1981}. In recent years, the applications of the random algebraic method to various extremal problems have appeared in several papers~\cite{Conlon2014, Matthew2019, Ma2018}.

 We next show the following upper bound of the classical Tur\'{a}n number $\text{ex}_{r}(n,K_{s,t}^{(r)})$ for $r\geqslant 3$ and $s\geqslant t\geqslant 2$, which is a generalization of the result of Mubayi and Verstra\"{e}te~\cite[Theorem 1.4]{Mubayi2004}.

\begin{theorem}\label{thm:CTUB}
  Let $s\geqslant t\geqslant 2$. Then
  \begin{equation*}
    \textup{ex}_{r}(n,K_{s,t}^{(r)})=O(n^{r-\frac{1}{t}}).
  \end{equation*}
\end{theorem}

Let $\mathcal{T}$ in Theorem \ref{thm:GTLB} be an edge. Combining Theorems \ref{thm:GTLB} and \ref{thm:CTUB}, we can obtain the following asymptotic order for Tur\'{a}n number of complete bipartite $r$-graphs.

\begin{corollary}\label{cor:KSTR}
Let $r\geqslant 3.$ For any positive integer $t,$ there exists some constant $c_{r,t}$ which depends on $r$ and $t$, such that when $s\geqslant c_{r,t},$ we have
  \begin{equation*}
    \textup{ex}_{r}(n,K_{s,t}^{(r)})=\Theta(n^{r-\frac{1}{t}}).
  \end{equation*}
\end{corollary}

If $\mathcal{T}$ is a complete bipartite $r$-graph $K_{a,b}^{(r)}$, where $a=1$ and $b<t$, then we obtain the asymptotic bound for generalized Tur\'{a}n number $\text{ex}_{r}(n,K_{a,b}^{(r)},K_{s,t}^{(r)})$.
\begin{theorem}\label{thm:generalKST}
 Let $r\geqslant 3.$ For any positive integer $t,$ there exists some constant $c'_{r,t}$ which depends on $r$ and $t$ such that if $s\geqslant c'_{r,t},$ $a=1$ and $b<t$, then we have
  \begin{equation*}
    \textup{ex}_{r}(n,K_{a,b}^{(r)},K_{s,t}^{(r)})=\Theta(n^{a+b(r-1)-\frac{ab}{t}}).
  \end{equation*}
\end{theorem}
In the simple graph case, there were several results shown in~\cite{Alon2016, Beka2019JGT, GerbnerJCTB2020, Gerbner2019EUJC} concerning the generalized Tur\'{a}n problems. However, in the hypergraph case, much less is known about the Tur\'{a}n numbers. Corollary~\ref{cor:KSTR} determines the asymptotic order for Tur\'{a}n numbers of complete bipartite $r$-graphs $K_{s,t}^{(r)}$ when $s$ is sufficiently larger than $t$. Moreover, the situation is even worse for the generalized hypergraph Tur\'{a}n problems, where we only know such tight results due to Ma, Yuan and Zhang~\cite{Ma2018}. Hence Theorem~\ref{thm:generalKST} provides some new tight results on the generalized hypergraph Tur\'{a}n problems.

In addition to the results mentioned above, we also consider the case when $\mathcal{H}$ is a complete $r$-partite $r$-graph $K_{s_{1},s_{2},\ldots,s_{r}}^{(r)}$, which can be seen as another extension of the complete bipartite graph. As we have mentioned, Mubayi~\cite{Mubayi2002} conjectured that $\text{ex}_{r}(n,K_{s_{1},s_{2},\ldots,s_{r}}^{(r)})=\Theta(n^{r-\frac{1}{\prod_{i=1}^{r-1}s_{i}}})$, where $s_{1}\leqslant s_{2}\leqslant\dots\leqslant s_{r}$. In particular, when $s_{1},s_{2},\ldots,s_{r}$ are relatively small, the case $\text{ex}_{r}(n,K_{s_{1},s_{2},\ldots,s_{r}}^{(r)})$ is more interesting. For example, Katz, Krop and Maggioni~\cite{Katz2002} showed that $\text{ex}_{3}(n,K_{2,2,2}^{(3)})=\Omega(n^{\frac{8}{3}}),$ which beats the lower bound from the probabilistic method. Next, we will show an improved lower bound for $\text{ex}_{3}(n,K_{2,2,7}^{(3)})$ as follows.
\begin{theorem}\label{thm:r=3}
  \[\textup{ex}_{3}(n,K_{2,2,7}^{(3)})\geqslant \frac{1}{27}n^{\frac{19}{7}}+o(n^{\frac{19}{7}}).\]
\end{theorem}
The best previously known lower bound $\text{ex}_{3}(n,K_{2,2,7}^{(3)})=\Omega(n^{\frac{73}{27}})$ was obtained by probabilistic methods. Theorem~\ref{thm:r=3} improves this by an explicit construction. Note that the upper bound part of Mubayi's conjecture was
proven by Erd\H{o}s~\cite{1964Erdos}, so the lower bound is the more interesting part.

The rest of this paper is organized as follows. In Section \ref{sec:bipartite}, we focus on the complete bipartite $r$-graphs. First we prove Theorem~\ref{thm:GTLB} via the random algebraic construction. Then we give some general upper bounds to derive Corollary~\ref{cor:KSTR} and Theorem~\ref{thm:generalKST} in Section~\ref{sec:Upperbound}. In Section~\ref{sec:degenerate}, we provide a new lower bound for $\text{ex}_{3}(n,K_{2,2,7}^{(3)})$. Section~\ref{sec:remarks} contains some remarks and the remaining problems on the main topics.

\section{Constructions for $K_{s,t}^{(r)}$-free $r$-graphs, $s>t$}\label{sec:bipartite}
 In this section, our goal is to prove Theorem \ref{thm:GTLB} via the random algebraic construction.

 \subsection{Random algebraic construction}
 As far as we know, usually there are two types of constructions as follows:
 \begin{enumerate}
   \item Randomized constructions with alternations, which are quite general and easy to apply, but usually do not give tight bounds.
   \item Algebraic constructions, which give tight bounds but appear to be somewhat magical and only work in certain special situations.
 \end{enumerate}
 We briefly review the related work in hypergraph Tur\'{a}n problem. Let $\mathcal{H}$ be an $r$-graph with $v$ vertices and $e$ edges. It was shown in \cite{Brown1973} that
  \begin{equation*}
  \textup{ex}_{r}(n,\mathcal{H})=\Omega(n^{\frac{er-v}{e-1}}).
  \end{equation*}
The above lower bound was obtained by a standard probabilistic argument. For example, when $\mathcal{H}=K_{s,t}^{(r)},$ the randomized construction gives a lower bound $\textup{ex}_{r}(n,K_{s,t}^{(r)})=\Omega(n^{r-\frac{1}{t}-\frac{(r-1)t^{2}-rt+1}{st^{2}-t}}).$

  Recently, there is an interesting idea of Bukh \cite{Bukh2015RAC} called ``random algebraic construction'', which combines these two approaches. The idea is to construct a graph with vertex set $V=\mathbb{F}_{q}^{s}\times \mathbb{F}_{q}^{s}$, just by choosing a random polynomial $f\in \mathbb{F}_{q}[x_{1},x_{2},\ldots,x_{s},y_{1},y_{2},\ldots,y_{s}]$ (within a certain family, say with bounded degree) and letting $(x,y)\in V$ be an edge if and only if $f(x,y)=0.$ The method aims to combine the advantages of both the flexibility of randomized constructions and the rigidity of algebraic constructions. Several papers \cite{Matthew2019, Tait2019SIDMA, Ma2018} developed this method and generalized the idea to hypergraphs.

 In order to apply the random algebraic method, our first task is to establish the relationship between polynomials and hypergraphs.

 For given positive integers $t,r$ with $r\geqslant 3$ and an $r$-graph $\mathcal{T}$ with $v$ vertices and $e$ edges, throughout this section, we always denote $d=(r-1)t^{2}-t+e+1$. Let $q$ be a sufficiently large prime power, and $\mathbb{F}_{q}$ be the finite field of order $q$.

 Let $\textbf{X}^{i}=(X_{1}^{i},X_{2}^{i},\ldots,X_{t}^{i})\in \mathbb{F}_{q}^{t}$ for each $i\in [r]$. Consider polynomials $f\in \mathbb{F}_{q}[\textbf{X}^{1},\textbf{X}^{2},\ldots,\textbf{X}^{r}]$ with $rt$ variables over $\mathbb{F}_{q}$. We say such a polynomial $f$ has degree at most $td$ in $\textbf{X}^{i},$ if each of its monomials has degree at most $td$ with respect to $\textbf{X}^{i}$, that is,  $(X_{1}^{i})^{\alpha_{1}}(X_{2}^{i})^{\alpha_{2}}\cdots(X_{t}^{i})^{\alpha_{t}}$ satisfies $\sum\limits_{j=1}^{t}\alpha_{j}\leqslant td.$ Moreover, a polynomial $f$ is called symmetric if exchanging $\textbf{X}^{i}$ with $\textbf{X}^{j}$ for every $1\leqslant i\leqslant j\leqslant r$ does not affect the value of $f$. For convenience, we can view the domain of symmetric polynomials as the family $\binom{\mathbb{F}_{q}^{t}}{r}$. Then given a symmetric polynomial $f$, we can define an $r$-graph $\mathcal{G}_{f}$ as following: the vertex set $V(\mathcal{G}_{f})$ is a copy of $\mathbb{F}_{q}^{t},$ and every $r$-tuple $\{u^{1},u^{2},\ldots,u^{r}\}\in \binom{V}{r}$ forms an edge of $\mathcal{G}_{f}$ if and only if $f(u^{1},u^{2},\ldots,u^{r})=0.$

 Let $\mathcal{P}\subseteq \mathbb{F}_{q}[\textbf{X}^{1},\textbf{X}^{2},\ldots,\textbf{X}^{r}]$ be the set of all symmetric polynomials of degree at most $td$ in $\textbf{X}^{i}$ for every $1\leqslant i \leqslant r.$ Then we choose a polynomial $f$ from $\mathcal{P}$ uniformly at random and let $\mathcal{G}=\mathcal{G}_{f}$ be the associated $r$-graph. Now we need to introduce two important lemmas from \cite{Bukh2015RAC} and \cite{Ma2018}. The first lemma is the key insight of the random algebraic construction, which provides very non-smooth probability distributions. While the second lemma will help us calculate the probability in certain situations.

 \begin{lemma}(\cite{Bukh2015RAC}, Lemma 5)\label{lemma:LB1}
   For every $t$ and $d$, there exists a constant $c>0$ such that the following holds: suppose $f_{1}(Y),f_{2}(Y),\ldots,f_{t}(Y)$ are $t$ polynomials on $\mathbb{F}_{q}^{t}$ of degree at most $td$, and consider the set
  \begin{equation*}
    W=\{y\in \mathbb{F}_{q}^{t}: f_{1}(y)=f_{2}(y)=\cdots=f_{t}(y)=0\}.
  \end{equation*}
  Then either $|W|<c$ or $|W|\geqslant q-c\sqrt{q}.$
 \end{lemma}

 \begin{lemma}(\cite{Ma2018}, Lemma 2.2)\label{lemma:LB2}
    Given a set $U\subseteq \binom{\mathbb{F}_{q}^{t}}{r},$ let $V\subseteq \mathbb{F}_{q}^{t}$ be the set consisting of all points appearing as an element of an $r$-tuple in $U$. Suppose that $\binom{|U|}{2}<q,$ $\binom{|V|}{2}<q$ and $|U|\leqslant td.$ If $f$ is a random polynomial chosen from $\mathcal{P}$, then
  \begin{equation*}
    \mathbb{P}[f(u^{1},u^{2},\ldots,u^{r})=0,  \forall \{u^{1},u^{2},\ldots,u^{r}\}\in U]=q^{-|U|}.
  \end{equation*}
 \end{lemma}

 With the above tools in hand, we are ready to prove Theorem \ref{thm:GTLB}.

\subsection{Proof of Theorem \ref{thm:GTLB}}
We choose a polynomial $f\in \mathcal{P}$ uniformly at random and let $\mathcal{G}$ be the associated $r$-graph $\mathcal{G}_{f}$. Let $n=q^{t}$ be the number of vertices in $\mathcal{G}$, where $q$ is sufficiently large. Though this result only holds when $q$ is a
prime power and $n=q^{t},$ it is a simple matter to use Bertrand's postulate to show that the same conclusion holds for all positive integers $n$.
We will show that on average this $\mathcal{G}$ contains many copies of $\mathcal{T}$ but very few copies of $K_{s,t}^{(r)}$, assuming $s$ is sufficiently large. Then we can use the alteration argument to obtain a subhypergraph $\mathcal{G'}$ which is $K_{s,t}^{(r)}$-free and $\mathcal{G'}$ still contains the expected number of copies of $\mathcal{T}$.

Since $\mathcal{T}$ has $v$ vertices and $e$ edges, it is easy to check that $\binom{v}{2}<q,$ $\binom{e}{2}<q$ and $e<t((r-1)t^{2}-t+e+1)=td.$ Then by Lemma \ref{lemma:LB2}, for given $v$ vertices, the probability that such $v$ vertices form a copy of $\mathcal{T}$ is equal to $\frac{1}{q^{e}}.$ Denote $X$ as the number of copies of $\mathcal{T}$ in $\mathcal{G}$, then the expectation
\begin{equation*}
  \mathbb{E}[X]= \Omega(\frac{1}{q^{e}}\binom{q^{t}}{v})=\Omega(q^{tv-e})=\Omega(n^{v-\frac{e}{t}}).
\end{equation*}
  Let $R$ be a fixed labeled copy of $K_{1,t}^{(r)}$, and we denote its vertices as $a$ and $u_{j}^{i}$'s for $1\leqslant j\leqslant t$ and $i\in [r-1]$ such that $u_{j}^{1},u_{j}^{2},\ldots,u_{j}^{r-1}$ form $t$ distinct $(r-1)$-tuples. Now fix any sequence of vertices $w_{j}^{i}$ for $1\leqslant j\leqslant t$ and $i\in [r-1]$ in $\mathcal{G}$. Let $W$ be the family of copies of $R$ in $\mathcal{G}$ such that $w_{j}^{i}$ corresponds to $u_{j}^{i}$ for all $1\leqslant j\leqslant t$ and $i\in [r-1]$. It is difficult to estimate $|W|$ directly, hence we consider the value of $|W|^{d}$. Note that $|W|^{d}$ counts the number of ordered collections of $d$ copies of $R$ from $W$, where these copies of $R$ may be the same. So each member of such collections can be an element $P$ in
\begin{equation*}
  \mathcal{K}:=\{K_{1,t}^{(r)},K_{2,t}^{(r)},\ldots,K_{d,t}^{(r)}\}.
\end{equation*}
For given $P\in \mathcal{K}$, denote $N_{d}(P)$ as the total number of all possible ordered collections of $d$ copies of $R\in W$, which could appear in $\mathcal{G}$ as a copy of $P$. It is easy to see that $N_{d}(P)=O(n^{|P|-t(r-1)})$, where $|P|$ is the number of vertices in $P$. Since the number of edges $e(P)=t(|P|-t(r-1))$ of $P$ is at most $td$ and $q$ is sufficiently large, by Lemma \ref{lemma:LB2}, the probability that a potential copy $P$ appears in $\mathcal{G}$ is $q^{-e(P)}$. Through the above analysis, we have
\begin{equation*}
  \mathbb{E}[|W|^{d}]=\sum\limits_{P\in \mathcal{K}}N_{d}(P)q^{-e(P)}=\sum\limits_{P\in \mathcal{K}}O(q^{t(|P|-t(r-1))})q^{-e(P)}=O(1).
\end{equation*}
Note that the set of unfixed vertices in $W$ consists of vertices $x\in \mathbb{F}_{q}^{t}$ satisfying the system of $t$ equations
\begin{equation*}
  f(w_{j}^{1},w_{j}^{2},\ldots,w_{j}^{r-1},x)=0
  \end{equation*}
  for $1\leqslant j\leqslant t.$ Because $f(w_{j}^{1},w_{j}^{2},\ldots,w_{j}^{r-1},\cdot)$ has degree at most $td,$ then by Lemma \ref{lemma:LB1}, either $|W|<c$ or $|W|\geqslant q-c\sqrt{q}\geqslant \frac{q}{10}$, where the value of $c$ depends on $t$ and $d$. By the Markov's inequality, we obtain that
  \begin{equation*}
    \mathbb{P}[|W|\geqslant c]=\mathbb{P}[|W|\geqslant \frac{q}{10}]=\mathbb{P}[|W|^{d}\geqslant (\frac{q}{10})^{d}]\leqslant \frac{\mathbb{E}[|W|^{d}]}{(\frac{q}{10})^{d}}=\frac{O(1)}{q^{d}}.
  \end{equation*}

A sequence of vertices $w_{j}^{i}$ for $1\leqslant j\leqslant t$ and $i\in [r-1]$ is called bad, if the corresponding set $W$ satisfies $|W|\geqslant c.$ Let $B$ be the number of bad sequences in $\mathcal{G}$, it follows that
\begin{equation*}
  \mathbb{E}[B]\leqslant [t(r-1)]!\binom{n}{t(r-1)}\frac{O(1)}{q^{d}}=O(q^{(r-1)t^{2}-d})=O(q^{t-e-1}).
\end{equation*}
Now we remove a vertex from each bad sequence to obtain a new hypergraph $\mathcal{G'}$, clearly $\mathcal{G'}$ does not contain any bad sequences, so $\mathcal{G'}$ is $K_{s,t}^{(r)}$-free for $s\geqslant c.$ Note that each vertex is in at most $O(n^{v-1})$ copies of $\mathcal{T}$ in $\mathcal{G}$, so the total number of copies of $\mathcal{T}$ removed is at most $ O(n^{v-1})\cdot B.$ Hence the expected number of the remaining copies of $\mathcal{T}$ in $\mathcal{G'}$ is at least
\begin{equation*}
  \Omega(n^{v-\frac{e}{t}})-\mathbb{E}[B]\cdot O(n^{v-1})=\Omega(n^{v-\frac{e}{t}}).
\end{equation*}
It is easy to check that the expected number of remaining vertices is $n-O(n^{1-\frac{e+1}{t}})=n-o(n).$

Therefore, for any $s\geqslant c,$ there exists a $K_{s,t}^{(r)}$-free $r$-graph with at most $n$ vertices and $\Omega(n^{v-\frac{e}{t}})$ copies of $\mathcal{T}$. This completes the proof of Theorem \ref{thm:GTLB}.

\section{Upper bound for $K_{s,t}^{(r)}$-free $r$-graphs}\label{sec:Upperbound}
The result in Theorem \ref{thm:GTLB} is intended to motivate our investigation of the matched upper bounds for some $r$-graphs $\mathcal{T}$. In this section, we will present matched upper bounds for $\text{ex}_{r}(n,K_{s,t}^{(r)})$ and $\text{ex}_{r}(n,K_{a,b}^{(r)},K_{s,t}^{(r)})$ under certain conditions.

\subsection{Upper bound for $\text{ex}_{r}(n,K_{s,t}^{(r)})$}

This upper bound for $\text{ex}_{r}(n,K_{s,t}^{(r)})$ can be seen as a generalization of the result of Mubayi and Verstra\"{e}te~\cite[Theorem 1.4]{Mubayi2004}. Before we prove Theorem \ref{thm:CTUB}, we need the following useful lemma of Erd\H{o}s and Kleitman~\cite{erdos1968}.

\begin{lemma}\label{lemma:r!/r}(\cite{erdos1968})
  Let $\mathcal{G}$ be an $r$-graph on $rn$ vertices. Then $\mathcal{G}$ contains an $r$-partite subhypergraph $\mathcal{G'}$, with all parts of size $n$, and $e(\mathcal{G'})\geqslant \frac{r!}{r^r}e(\mathcal{G})$.
\end{lemma}

 We write $z(n,K_{s,t}^{(r)})$ for the maximum number of edges in an $r$-partite $K_{s,t}^{(r)}$-free $r$-graph in which all parts have size $n$. By Lemma \ref{lemma:r!/r}, it suffices to prove that $z(n,K_{s,t}^{(r)})=O(n^{r-\frac{1}{t}})$.

\begin{proof}[Proof of Theorem \ref{thm:CTUB}]
 Let $A_{1},A_{2},\ldots,A_{r-1},B$ be the $r$ parts of size $n$ of an $r$-partite $K_{s,t}^{(r)}$-free $r$-graph $\mathcal{H}$. Suppose that $\mathcal{H}$ has more than $c_{s,t}'n^{r-\frac{1}{t}}$ edges, where $c_{s,t}'$ is defined as the smallest integer for which every bipartite graph with parts $X$ and $Y$ of size $n$ having more than $c_{s,t}'n^{2-\frac{1}{t}}$ edges contains a $K_{s,t}$ with $t$ vertices in $X$ and $s$ vertices in $Y$. Clearly $c_{s,t}'$ is independent of $n$ by the K\"{o}v\'{a}ri-S\'{o}s-Tur\'{a}n bound~\cite{Kovari1954}. Consider the complete $(r-1)$-partite $(r-1)$-uniform hypergraph $K_{n,n,\ldots,n}^{(r-1)}$ on vertex set $A_{1}\times A_{2}\times\cdots\times A_{r-1}$. It was shown in~\cite{1975matching} that $K_{n,n,\ldots,n}^{(r-1)}$ has a perfect matching decomposition, hence we can partition the $(r-1)$-tuples of $A_{1}\times A_{2}\times\ldots\times A_{r-1}$ into $\frac{n^{r-1}}{n}=n^{r-2}$ matchings $M_{1},M_{2},\ldots,M_{n^{r-2}}$. Let $\mathcal{H}_{i}$ be the subhypergraph of $\mathcal{H}$ induced by those edges that contain some $(r-1)$-tuples of $M_{i}$. By the pigeonhole principle, there exists some $i$ such that $\mathcal{H}_{i}$ contains more than $c_{s,t}'n^{2-\frac{1}{t}}$ edges. Next we construct an auxiliary bipartite graph $G_{i}$ on vertex set $A_{r-1}\cup B$, with edge set $$\{(a_{r-1},b):\exists (a_{1},a_{2},\ldots,a_{r-2}),a_{i}\in A_{i}, (a_{1},a_{2},\ldots,a_{r-1},b)\in E(\mathcal{H}_{i})\}.$$ Then by the choice of $c'_{s,t}$, we conclude that $G_{i}$ contains a copy of $K_{s,t}$ with $s$ vertices in $B$ and $t$ vertices in $A_{r-1}$, which extends via $M_{i}$ to a $K_{s,t}^{(r)}$ in $\mathcal{H}$. The proof is finished.
\end{proof}

\subsection{Upper bound for $\text{ex}_{r}(n,K_{a,b}^{(r)},K_{s,t}^{(r)})$}
We now complete the proof of Theorem \ref{thm:generalKST} by proving the following lemma. The main idea of this proof is based on the ideas in~\cite{Alon2016}
\begin{lemma}
  If $b<t$, then we have
  \begin{equation*}
    \textup{ex}_{r}(n,K_{1,b}^{(r)},K_{s,t}^{(r)})=O(n^{b(r-1)-\frac{b}{t}+1}).
  \end{equation*}
\end{lemma}

\begin{proof}
Let $\mathcal{G}$ be a $K_{s,t}^{(r)}$-free $r$-graph with $n$ vertices. For each vertex $v\in\mathcal{G}$, let $N(v)$ be the following set
\begin{equation*}
  N(v)=\{(b_{1},b_{2},\ldots,b_{r-1})| b_{i}\in V(\mathcal{G}), (v,b_{1},b_{2},\ldots,b_{r-1})\in E(\mathcal{G})\}.
\end{equation*}
It is easy to see the number of $K_{1,b}^{(r)}$ in $\mathcal{G}$ is at most
\begin{align*}
  \sum\limits_{v\in V(\mathcal{G})}\binom{|N(v)|}{b}&\leqslant \frac{1}{b!}\sum\limits_{v\in V(\mathcal{G})}|N(v)|^{b}\\
  &\leqslant \frac{1}{b!}n^{1-\frac{b}{t}}(\sum\limits_{v\in V(\mathcal{G})}|N(v)|^{t})^{\frac{b}{t}}\\
  &\leqslant (1+o(1))\frac{n^{1-\frac{b}{t}}}{b!}\bigg(\big((s-1)(t(r-1))!+(t(r-1)-1)!\big)n^{t(r-1)}\bigg)^{\frac{b}{t}}\\
  &=O(n^{b(r-1)-\frac{b}{t}+1}).
\end{align*}
We need some basic facts in the above estimation. The first is that for any $0<p \leqslant q,$ $\sum\limits_{i=1}^{m}x_{i}^{p}\leqslant m^{1-\frac{p}{q}}(\sum\limits_{i=1}^{m}x_{i}^{q})^{\frac{p}{q}}.$ Moreover, we estimate $\sum\limits_{A}\binom{|N(v)|}{t}$ via double counting. That is, we take advantage of the formulation $\sum\limits_{A}\binom{|N(v)|}{t}=\sum\limits_{T\in\mathcal{T}_{1}}|N(T)|+\sum\limits_{T\in\mathcal{T}_{2}}|N(T)|$, where $\mathcal{T}_{1}$ consists of all $t$ vertex disjoint $(r-1)$-tuples and $\mathcal{T}_{2}$ consists of the other $t$ $(r-1)$-tuples. Moreover, $N(T)$ consists of the vertices which are adjacent to every $(r-1)$-tuple in $T$. Consider the first part $\sum\limits_{T\in\mathcal{T}_{1}}|N(T)|$, for every $T\in\mathcal{T}_{1}$, if there are more than $(s-1)$ vertices in $N(T)$, then we can obtain a copy of $K_{s,t}^{(r)}$, which is a contradiction. For the second part $\sum\limits_{T\in\mathcal{T}_{2}}|N(T)|$, note that $|N(T)|<n$ and the number of vertices in $T$ is less than $t(r-1)$, thus we have $\sum\limits_{T\in\mathcal{T}_{2}}|N(T)|<(1+o(1))(t(r-1)-1)!n^{t(r-1)}$. The proof is finished.
\end{proof}

\section{$\text{ex}_{3}(n,K_{2,2,7}^{(3)})$}\label{sec:degenerate}
Let $\mathcal{H}$ be an $r$-graph with $v$ vertices and $e>0$ edges. An application of the probabilistic method shows that $\text{ex}(n,\mathcal{H})>cn^{\alpha}$, where $\alpha=r-\frac{v-r}{e-1}$ and $c$ is independent of $n$ \cite{Brown1973}. This yields $\text{ex}(n,K_{2,2,7}^{(3)})=\Omega(n^{\frac{73}{27}})$. In this section, we improve the exponent $\frac{73}{27}$ to $\frac{19}{7}$ by proving Theorem~\ref{thm:r=3}. Our construction is a variation of norm hypergraphs, thus the construction is explicit.

  Let $\mathbb{F}_{q}$ be a finite field, and $\mathbb{F}_{q^{r}}$ be a finite field extension of $\mathbb{F}_{q}$, the \emph{norm} $\textup{Norm}_{r}(x)$ of $x\in\mathbb{F}_{q^{r}}$ over $\mathbb{F}_{q}$ is defined by
  \[\textup{Norm}_{r}(x)=\prod_{i=0}^{r-1}x^{q^{i}}.\]
  Then $\textup{Norm}_{r}(x)\in\mathbb{F}_{q}$. The following result can be found in \cite{KRS96}.
\begin{lemma}(\cite{KRS96})\label{lem}
If $(D_{1},d_{1}),\dots,(D_{s},d_{s})$ are distinct elements of $\mathbb{F}_{q^{s-1}}\times\mathbb{F}_{q}^{*}$, then the system of $s$ equations
\[\textup{Norm}_{s-1}(D_{i}+X)=d_{i}x,\ 1\leqslant i\leqslant s\]
has at most $(s-1)!$ solutions $(X,x)\in\mathbb{F}_{q^{s-1}}\times\mathbb{F}_{q}^{*}$.
\end{lemma}
We also need the following lemma.
\begin{lemma}\label{lem4}
Let $m$ be a sufficiently large integer, $k=\lfloor\sqrt{m}\rfloor-1$ and $\ell=\lfloor\frac{k}{2}\rfloor$. Let
\begin{align*}
&S_{1}=\{0,1,2,\dots,\ell-1\},\\
 &S_{2}=\{0,k,2k,\dots,(\ell-1)k\},\\
  &S_{3}=\{0,k+1,2(k+1),\dots,(\ell-1)(k+1)\}
\end{align*}
 be additive sets in $\mathbb{Z}_{m}$. Then $|S_{i}+S_{j}|=|S_{i}||S_{j}|=\ell^{2}$ for $1\leqslant i\neq j\leqslant 3$.
\end{lemma}
\begin{proof}
It is easy to see that $|S_{1}+S_{2}|=\ell^{2}$.

For any $x\in \mathbb{Z}_{m}$, if $x=i+j(k+1)$, where $0\leqslant i,j\leqslant \ell-1$. Then $x=(i+j)+jk$, there is at most one solution for $i,j$. Hence $|S_{1}+S_{3}|=\ell^{2}$. Similarly, we have $|S_{2}+S_{3}|=\ell^{2}$.
\end{proof}

\begin{proof}[Proof of Theorem~\ref{thm:r=3}]
Let $q$ be an odd prime power, $k=\lfloor\sqrt{q-1}\rfloor-1$ and $\ell=\lfloor\frac{k}{2}\rfloor$. Let \begin{align*}
&S_{1}=\{0,1,2,\dots,\ell-1\},\\
 &S_{2}=\{0,k,2k,\dots,(\ell-1)k\},\\
  &S_{3}=\{0,k+1,2(k+1),\dots,(\ell-1)(k+1)\}
\end{align*} be additive sets in $\mathbb{Z}_{m}$. By Lemma~\ref{lem4}, $|S_{i}+S_{j}|=|S_{i}||S_{j}|=\ell^{2}$ for $1\leqslant i\neq j\leqslant 3$.

Let $g$ be a primitive element of $\mathbb{F}_{q}$, and $B_{i}=\{g^{j}: j\in S_{i}\}$ for $1\leqslant i\leqslant 3$. Let $\mathcal{G}$ be a $3$-graph with parts $A_{i}=\mathbb{F}_{q^{3}}\times B_{i}$, $i=1,2,3$. The vertices $(D_{i},d_{i})\in A_{i}$, $i=1,2,3$ form an edge if $\text{Norm}_{3}(D_{1}+D_{2}+D_{3})=d_{1}d_{2}d_{3}$.

Clearly $\mathcal{G}$ has $n:=3\ell q^{3}=\frac{3}{2}q^{\frac{7}{2}}+o(q^{\frac{7}{2}})$ vertices, and it is easy to count that there are $\frac{(\ell q^{3})^{3}}{q}\geqslant \frac{1}{27}n^{\frac{19}{7}}+o(n^{\frac{19}{7}})$ edges.

We claim that $\mathcal{G}$ is $K_{2,2,7}^{(3)}$-free. Assume to the contrary, there exists a copy of $K_{2,2,7}^{(3)}$ in $\mathcal{G}$. Without loss of generality, suppose that $(D_{i},d_{i})\in A_{1}$, $(E_{j},e_{j})\in A_{2}$, $(X_{k},x_{k})\in A_{3}$, $i,j\in [2]$, $k\in [7]$ form a copy of $K_{2,2,7}^{(3)}$. Let $T_{ij}=D_{i}+E_{j}$ and $t_{ij}=d_{i}e_{j}$. Then we have
\[\text{Norm}_{3}(T_{ij}+X_{k})=t_{ij}x_{k}\]
for $i,j\in[2]$ and $k\in[7]$. This also implies that the system of equations
\[\text{Norm}_{3}(T_{ij}+X)=t_{ij}x\]
for $i,j\in[2]$ has at least $7$ solutions for $(X,x)$.

By the definition of $B_{i}$, we have $|\{t_{ij}: i,j\in[2]\}|=4$. Hence $(T_{ij},t_{ij})$, $i,j\in[2]$ are distinct elements. By Lemma~\ref{lem}, there are at most $6$ solutions for such a system of equations, which is a contradiction. Thus, $\mathcal{G}$ is $K_{2,2,7}^{(3)}$-free.
\end{proof}

\begin{remark}
We believe that the exponent $\frac{19}{7}$ can be improved, hence we made no attempt to optimize the leading coefficient $\frac{1}{27}$ in the proof above.
\end{remark}
\section{Concluding remarks}\label{sec:remarks}
In this paper, we have studied two extensions of hypergraph Tur\'{a}n problems of complete bipartite graphs. The first object is the complete bipartite $r$-graph. The authors in \cite{Mubayi2004} introduced this structure and gave some general bounds and constructions for $\text{ex}_{r}(n,K_{s,t}^{(r)})$. They also presented a conjecture for $3$-graphs. Here we generalize their conjecture for $r\geqslant 3.$
\begin{conjecture}\label{conj:mubayi}
  Let $s,t$ be integers with $2\leqslant s\leqslant t,$ then
  \begin{equation*}
    \textup{ex}_{r}(n,K_{s,t}^{(r)})=\Theta(n^{r-\frac{2}{s}}).
  \end{equation*}
\end{conjecture}

Though we still can not verify this conjecture, there is some evidence that supports this conjecture. For example, Ergemlidze, Jiang and Methuku~\cite{JiangJCTA2020} showed that $\text{ex}_{4}(n,K_{2,t}^{(4)})\geqslant (1+o(1))\frac{t-1}{8}n^{3}.$

Moreover, in~\cite{Mubayi2004} the authors remarked that their results can apply for both $t\geqslant s$ and $s>t,$ hence for simplicity they let $t\geqslant s.$  However, when $s$ is sufficiently larger than $t$, to our surprise, we obtain the matched lower bounds for $\text{ex}_{r}(n,K_{s,t}^{(r)})=\Omega(n^{r-\frac{1}{t}})$ via the random algebraic construction.

 We also obtain the lower bounds for generalized \turan number $\text{ex}_{r}(n,\mathcal{T},K_{s,t}^{(r)}),$ and we show the matched upper bounds when $\mathcal{T}$ is an edge or a complete bipartite $r$-graph $K_{1,b}^{(r)}$ with $b<t$. It will interesting to find more examples reaching the lower bounds.

\section*{Acknowledgements}
The authors would like to thank Dr.~Jie Ma and Dr.~Chong Shangguan for their helpful comments and express their gratitude to the anonymous
reviewers for their detailed and constructive comments which are
very helpful to the improvement of the technical presentation of
this paper.

\end{document}